\documentclass[bsl]{asl}
\usepackage{amssymb,stmaryrd,color,nicefrac,xcolor}
\newtheorem{theorem}{Theorem}[section]
\newtheorem{lemma}[theorem]{Lemma}
\usepackage[inline]{enumitem}

\newtheorem{corollary}[theorem]{Corollary}
\newtheorem{proposition}[theorem]{Proposition}

\newlist{Cases}{enumerate}{9}
\setlist[Cases,1]{label={\sc Case {\rm \arabic*}},wide, labelwidth=!, labelindent=0pt}
\setlist[Cases,2]{label*= {\rm .\arabic*},wide, labelwidth=!, labelindent=0pt}
\setlist[Cases,3]{label*= {\rm .\arabic*},wide, labelwidth=!, labelindent=0pt}
\setlist[Cases,4]{label*= {\rm .\arabic*},wide, labelwidth=!, labelindent=0pt}
\setlist[Cases,5]{label*= {\rm .\arabic*},wide, labelwidth=!, labelindent=0pt}
\setlist[Cases,6]{label*= {\rm .\arabic*},wide, labelwidth=!, labelindent=0pt}
\setlist[Cases,7]{label*= {\rm .\arabic*},wide, labelwidth=!, labelindent=0pt}
\setlist[Cases,8]{label*= {\rm .\arabic*},wide, labelwidth=!, labelindent=0pt}
\setlist[Cases,9]{label*= {\rm .\arabic*},wide, labelwidth=!, labelindent=0pt}

\makeatletter
\newcommand*{\ineq}[2][]{%
  \begingroup
    \refstepcounter{equation}%
    \ifx\\#1\\%
    \else
      \label{#1}%
    \fi
    \relpenalty=10000 %
    \binoppenalty=10000 %
    \@eqnnum \ \ensuremath{%
      \hskip 20pt #2%
    }%
  \endgroup
}
\makeatother

\newcommand{\done}{\color{blue}}

\newcommand{\longproof}[1]{{\color{orange}#1}}
\renewcommand{\longproof}[1]{}

\newcommand{\shortproof}[1]{{#1}}

\renewcommand{\done}{}

\newcommand{\reread}{\color{orange}}

\renewcommand{\reread}{}

\newcommand{\mchb}[3]{{\uparrow}^{#1} _{#2} #3}

\newcommand{\good}[3]{{\rm G}^{{\mathcal {#1}}}_{#3}(#2)}

\newcommand{\goodt}[2]{{\rm G}_\infty^{\mathcal #1} (#2)}

\newcommand{\goodp}[4]{{\rm G}^{\mathcal #1}_{#3}(#2|#4)}

\newcommand{\goodpt}[3]{{\rm G}^{\mathcal #1}_{\infty}(#2|#3)}

\newcommand{\norm}[1]{\|#1\|}
\newcommand{\nf}[2]{\text{\sc nf}_{#1}(#2)}
\newcommand{\val}[1]{|#1|}

\newcommand{\bch}[2]{ {\uparrow #2  } #1 }
\newcommand{\bcn}[2]{ {\uparrow   } #1 }

\newcommand{\nnf}{=_k}

\newcommand{\gknfsym}{=}

\newcommand{\nfpar}[1]{\gknfsym_{#1}}

\newcommand{\gknf}{\nfpar{k} }

\newcommand{\fs}[2]{#1[#2]}

\newcommand{\mc}[1]{{\rm mc}(#1)}
\newcommand{\bases}[2]{^{#2}_{#1}}

\newcommand{\putaway}[1]{}

\newcommand{\ve}{\varepsilon}

\newcommand{\al}{{\alpha}}

\newcommand{\be}{{\beta}}

\newcommand{\om}{\omega}

\newcommand{\longversion}[1]{}
\theoremstyle{definition}
\newtheorem{definition}[theorem]{Definition}
\newtheorem{example}[theorem]{Example}

\theoremstyle{remark}
\newtheorem{remark}[theorem]{Remark}

\numberwithin{equation}{section}

\begin{document}

\title{Complete Intuitionistic Temporal Logics for Topological Dynamics}

\author[Fern\'andez and Weiermann]{David Fern\'{a}ndez-Duque${}^{1}$ \and Andreas Weiermann${}^{2}$}
\address[1]{Department of Philosophy\\
University of Barcelona\\
\email{Fern\'andez-Duque}{fernandez-duque@ub.edu}}
\address[2]{Department of Mathematics: Analysis, Logic and Discrete Mathematics\\
Ghent University\\
\email{Weiermann}{Andreas.Weiermann@UGent.be}}

\title{A Walk with Goodstein}


%    author one information

%    \subjclass is required.


%\date{}

%\dedicatory{}

%    Abstract is required.

\done
\begin{abstract}
Goodstein's principle is arguably the first purely number-theoretic statement known to be independent of Peano arithmetic.
It involves sequences of natural numbers which at first appear to diverge, but eventually decrease to zero.
These sequences are defined relative to a notation system based on exponentiation for the natural numbers.
In this article, we provide a self-contained and modern analysis of Goodstein's principle, obtaining some variations and improvements.
We explore notions of optimality for notation systems and apply them to the classical Goodstein process and to a weaker variant based on multiplication rather than exponentiation.
In particular, we introduce the notion of {\em base-change maximality,} and show how it leads to far-reaching extensions of Goodstein's result.
We moreover show that by varying the initial base of the Goodstein process, one readily obtains independence results for each of the fragments ${\sf I}\Sigma_n$ of Peano arithmetic.
\end{abstract}

\maketitle

\section{Introduction}

Ever since G\"odel's first incompleteness theorem \cite{Godel1931}, we know that Peano arithmetic ($\sf PA$) cannot prove every true arithmetical statement.
However, G\"odel's proof is based on a specifically constructed statement that can be argued to be artificial from the perspective of mainstream mathematics.
Since then, several facts of a purely combinatorial or number-theoretic nature have been shown to be independent from $\sf PA$ \cite{
Clote,Erickson,Hajek,Kanamori,KirbyFlipping,ParisHarrington}, but the oldest example is a theorem of Goodstein \cite{Goodstein1944}, although it was only shown to be independent much later by Kirby and Paris~\cite{Kirby}; see~\cite{Rathjen2015} for a historical overview.
Goodstein's result will be the focus of this work.

Informally, one writes a natural number $m$ in hereditary base $2$, meaning that $m$ is represented in base $2$ in the usual way, then so is each exponent that appears, and so on. A precise definition will be given later, but for example, $m=20$ would be written as $2^{2^2} + 2^2 $.
The Goodstein process based on $m$ is a sequence $(\good {{}}mi )_{i <\alpha}$ with $\alpha \leq \infty$, such that $\good {{}}m0 = m$ and, if $\good {{}}mi $ is defined and positive, $\good {{}}m{i+1}$ is obtained by first writing $ \good {{}}mi $ in hereditary base $i+2$, then replacing every instance of $i+2$ by $i+3$, and finally subtracting $1$.
The sequence terminates if it reaches zero.
Thus for example, 
\[\good {{}}{20}1 = 3^{3^3}  +3^3 - 1 = 3^{3^3} +  3^2\cdot 2 + 3\cdot 2 +2.\]
This number is already large enough to be rather cumbersome to write out and, in fact, the sequence will grow very rapidly for some time.
This should make Goodstein's principle quite surprising: for any $m$ that we start with, there will be a value of $i$ such that $\good {{}}mi=0$.
The proof uses transfinite induction, and Kirby and Paris showed that this was, in a precise sense, unavoidable, leading to unprovability in $\sf PA$ \cite{Kirby}.

A natural question to ask is if this particular way of writing natural numbers is `canonical' in some way. For example, we could just as well have written $20=2^{2^2} + 2 + 2 $.
This would lead to a different candidate for $\good {}{20}1$; namely, $3^{3^3} + 3 +   2 $.
Is there some sense in which the standard representation of $20$ is preferable?
Will the Goodstein process still terminate if we choose a different representation of each natural number?

Remarkably, the answer to both of these questions is `yes'.
In fact, the two questions are intimately connected, as we will see throughout the paper.
Regarding the first question, we identify two potential criteria for a canonical system of normal forms: first, it could be {\em norm minimizing,} meaning that we use the least possible number of symbols to write a number.
Second, it could be {\em base-change maximal,} which roughly states that $\good {} m {i+1}$ will be as large as possible given $\good {}mi$.
While the first property is arguably more intuitive, the second turns out to be surprisingly useful.
In particular, termination for a Goodstein process based on a base-change maximal notation system implies that any other notation system (based on the same primitive functions) will also yield a terminating Goodstein process.

As we will see, the hereditary exponential normal form for natural numbers enjoys both norm minimization and base-change maximality.
This tells us that every {\em Goodstein walk} is finite, by which we mean every sequence of numbers $(m_i)_{i=0}^\alpha$, where $m_{i+1}$ is obtained by writing $m_i$ in an arbitrary fashion using addition and base-$(i+2)$ exponentiation, then replacing every instance of $i+2$ by $i+3$ and subtracting one.
We may even use multiplication, which is not needed for Goodstein's original theorem, even though norm minimality fails when multiplication is involved.
We will formalize and prove these results in Section~\ref{secWalks}.

We also consider some variants of Goodstein's principle of lower proof-theoretic strength.
For each $n\geq 1$, recall that ${\sf I}\Sigma_n$ is the fragment of $\sf PA$ where induction is restricted to $\Sigma_n$ formulas.
First, we consider notations based on addition and multiplication, but not exponentiation.
This leads to an independence result for ${\sf I}\Sigma_1$; see Section~\ref{secWeak} for details.
Finally, we show that by varying the initial base (i.e., rather than writing $m$ in hereditary exponential base $2$, we use a different base $r\geq 2$) but restricting $m$, we may obtain a parametrized version of Goodstein's principle which provides independence results for each ${\sf I}\Sigma_n$ with $n\geq 1$.
These parametrized Goodstein principles are detailed in Section~\ref{secPhase}.
Meskens and Weiermann~\cite{MeskensWeiermann} have also shown independence results for ${\sf I}\Sigma_n$ based on Goodstein principles, albeit our approach is quite different: we consider Goodstein processes where only the initial base is modified, while they consider sequences with slowly changing bases.
Our presentation is mostly self-contained, so that in particular Goodstein's original theorem and its independence from Peano arithmetic are obtained via our methods.
\medskip

\paragraph{\bf Layout.} In Section \ref{secClassical} we review Goodstein's classical result and set up an abstract framework which sets the stage for generalizations.
Section \ref{secOptimal} then introduces the notions of norm minimality and base-change maximality, which will be a focus of the paper.
With these notions in mind, the following sections study various Goodstein processes: Section \ref{secWeak} considers a weakened Goodstein principle,
Section \ref{secExponential} studies the optimality of hereditary exponential normal forms, and Section \ref{secElementary} shows that base-change maximality holds even if we extend the notation system to include multiplication.
Section \ref{secLowerBound} then compares the termination time of Goodstein processes to Hardy functions, from which termination and independence is obtained.
The optimality results obtained are used in Section \ref{secWalks} to provide generalizations of Goodstein's theorem in terms of {\em Goodstein walks,} and Section~\ref{secPhase} parametrizes Goodstein's result to obtain provability phase transitions for each ${\sf I}\Sigma_n$.

\section{The Classical Goodstein Process}\label{secClassical}

Let us discuss the original Goodstein principle from an abstract perspective, which will be useful in the rest of the text.
A {\em notation system} is a family of function symbols $\mathcal F$ so that each $f\in \mathcal F$ is equipped with an arity $n_f>0$ and a function $\val f \colon \mathbb N^{n_f} \to \mathbb N$.
For a function symbol $f(x_0,\ldots,x_{n})$ of arity $n+1$, the parameter $x_0$ will be regarded as a `base' and usually denoted $k$ or $\ell$.

Given fixed $k\geq 2$, the set of {\em (closed) base $k$ terms,} $\mathbb T^\mathcal F_k$, is defined inductively so that if $\tau_1,\ldots,\tau_n \in \mathbb T^\mathcal F_k$ are terms and $f$ is a function symbol with arity $n+1$, then $f(k,\tau_1,\ldots,\tau_n ) \in \mathbb T^\mathcal F_k$.
We write $\mathbb T^\mathcal F$ for $\bigcup_{k=2}^\infty \mathbb T^\mathcal F_k$; note that $\mathbb T^\mathcal F$ contains terms of all bases, but each term has a unique base.\footnote{Goodstein processes with mixed bases could also be of interest and may be considered in future work.}
The {\em value} of a term $\tau = f(k,\tau_1,\ldots,\tau_n)$ is defined recursively by $\val \tau = \val f (k,\val{\tau_1},\ldots,\val{\tau_n})$.
The {\em norm} of $\tau$ is defined inductively by $\norm \tau = 1 + \sum_{i = 1}^n\norm{\tau_i}$; note that function symbols that depend only on $k$ have norm one.
In practice, we may also include constant symbols $c$, but for theoretical purposes these will be regarded as function symbols $f$ of arity one such that $|f|(k)\equiv c$.
Similarly, operations such as addition might not display $k$ but a term $\tau +\sigma$ is `officially' $f(k,\tau,\sigma)$ for some symbol $f$ with $|f|(k,x,y)=x+y$. 
%On occasion we also consider terms $ \tau (\vec x)$ in which $\vec x$ is a tuple of variables, in which case we call $\tau (\vec x)$ a {\em dependent term.}

%%%%BOOKMARK

It is important to make a conceptual distinction between function symbols and the functions they represent, as for example we can do induction on the complexity of a term, independently of its numerical value.
However, often we will not make a {\em notational} distinction and omit $|\cdot|$; whether an expression should be treated as a value or a term will be made clear from context.
For the classical Goodstein process, we will work with the functions/function symbols $0,x+y$ and $k^x$; we denote this notation system by $\mathcal E$ for `exponential', and write $\mathbb E_k$ instead of $\mathbb T^\mathcal E_k$ and $\mathbb E$ instead of $\mathbb T^\mathcal E$.

It is not required that each natural number have a single notation, but a canonical one may be chosen nonetheless.
A {\em normal form assignment} for a notation system $\mathcal F$ is a function $\nf \cdot \cdot \colon [2,\infty)\times \mathbb N \to \mathbb T^\mathcal F $ such that $\nf k n\in\mathbb T^\mathcal F_k $ and $\val{\nf k n} = n$ for all $k\geq 2$ and $n$.
A notation system equipped with a normal form assignment is called a {\em normalized notation system.}
In the case of $\mathbb E_k$, the normal form for $n\in\mathbb N$ is defined as follows.
Set $\nf k 0 = 0$.
For $n>0$, assume that $\nf k m$ is defined for all $m<n$.
Let $r$ be the unique natural number such that $k^r\leq n < k^{r+1}$, and $ b = n - k^r$.
Then set $\nf k n  = k^{\nf k  r } + \nf k  b  $.

Finally, we need a {\em base change} operation to define the Goodstein process.
Given $k\leq \ell$ and $\tau \in \mathbb T^\mathcal F_k$, we define $\bch\tau{^\ell}  $ recursively by
\[\bch{f(k,\tau_1,\ldots,\tau_n)}{^\ell} = f(\ell, \bch {\tau_1} {^\ell},\ldots, \bch {\tau_n} {^\ell}).\]

If a normal form assignment is given, we can extend operations on terms to natural numbers by first computing their normal form.
In particular, we define $\norm n_k = \norm{\nf  k n }$ and $\bch n{\bases k\ell} = |\bch {\nf k n}{^\ell}|$. To ease notation, we will sometimes write $\bch n{}$ instead of $\bch n {\bases{k}\ell}$, when $k$ and $\ell$ are made clear.
We may also write $\nf{k}\tau$ instead of $\nf k{\val \tau}$.

\begin{definition}\label{defGoodstein}
Let $\mathcal F$ be a normalized notation system and $m\in \mathbb N$.
We define the {\em $\mathcal F$-Goodstein sequence beginning at $m$ with initial base $r$} to be the unique sequence $( m_i )_{i <\alpha}$, where $\alpha \leq \infty$, so that
\begin{enumerate}

\item $m_0=m$,

\item $m_{i+1}  = \bch{m_i}{^{r+i+1}_{r+i}} - 1$ if $m_i > 0 $,

\item $\alpha = i+1 $ if $m_i = 0$; if there is no such $i$, then $\alpha = \infty$.

\end{enumerate}
We write $ \goodp Fmir:=m_i$, and we often write $\good Fmi$ instead of $\goodp Fmi2$.
\end{definition}

With this, we can state Goodstein's principle within our general framework.

\begin{theorem}[Goodstein]\label{theoGood}
For every $m\in \mathbb N$, there is $i\in\mathbb N$ such that $\good Emi = 0$.
\end{theorem}

This theorem is a consequence of Theorem~\ref{theoGrowthExp}, which we will state and prove later.
Note that in Definition~\ref{defGoodstein}, the normal forms used are essential when computing $\bch{m_i}{^{r+i+1}_{r+i}}$, and Goodstein sequences based on different normal forms may have wildly different behaviour.
However, as we will see in Section~\ref{secWalks}, Theorem~\ref{theoGood} is remarkably robust and holds true for {\em any} choice of normal forms.
The proof of this uses base-change maximality, a notion of `optimality' of normal forms, as we discuss next.

\section{Optimality Criteria for Normal Forms}\label{secOptimal}

For a given notation system $\mathcal F$, there may be many ways to assign normal forms to natural numbers.
The question thus arises: is there an `optimal' way to define normal forms?
The following two criteria could help answer this question.
We say that a normal form assignment  {\sc nf}  is:

\begin{itemize}
\item {\em norm minimizing} if whenever $k\geq 2$ and $\tau \in \mathbb T^\mathcal F_k$, it follows that $\norm \tau\geq \norm {\nf k \tau} $;

\item {\em base-change maximal} if whenever $k\geq 2$ and $\tau \in \mathbb T^\mathcal F_k$, it follows that $\lvert \bch \tau {^\ell}\rvert \leq \val{ \bch {\nf k\tau} {^\ell} } $ for all $\ell \geq k$.

\end{itemize}
The motivation for norm minimizing normal forms should be clear, as these provide the most succinct way to represent natural numbers.
Base-change maximality is perhaps a less obvious criterion, although the intuition is that we are using the fastest-growing functions available in order to represent numbers; from this perspective, one may expect that the two notions will often coincide (although not always).
Moreover, as we will see, base-change maximal normal forms are rather useful.
For one thing, under some mild assumptions, they satisfy a natural monotonicity property.

\begin{proposition}\label{propMaxToMon}
Let $\mathcal F $ be a normalized notation system which includes addition and a term $1$ which does not depend on $k$.
Suppose that $\mathcal F $ is base-change maximal.
Then, whenever $2\leq k<\ell$ and $m<n $, it follows that $\bch m{\bases k\ell} < \bch n{\bases k\ell}$.
\end{proposition}

\proof
Working inductively, we may assume that $n=m+1$.
Then, we have that $n=\val {\nf km+1}$, and by base-change maximality,
\[ \bch m{\bases k\ell} < \bch {\nf km}{^\ell}+1 = \bch {(\nf km + 1)}{^\ell} \leq \bch {\nf kn}{^\ell} = \bch n{\bases k\ell}.  \]
\endproof

In fact, this monotonicity property is crucial for proving that Goodstein processes terminate, and Proposition \ref{propMaxToMon} tells us that we have this property for free, given base-change maximality.

\begin{remark}\label{remMaxToMon}
Note that Proposition \ref{propMaxToMon} can also be applied `locally': if we know that $\mathcal F $ is base-change maximal whenever $\val \tau <N$ for some fixed value of $N$, then from $m<n<N$ we can deduce that $\bch m{\bases k\ell} < \bch n{\bases k\ell}$.
This restricted version will be useful in inductive arguments.
\end{remark}

In the sequel, we will evaluate various Goodstein-like processes according to these criteria.
We begin by considering a weak variant of Goodstein's original result.

\section{A Weak Goodstein Principle}\label{secWeak}

In this section we apply our framework to a weak Goodstein principle based on the `multiplicative' notation system $\mathcal M$, whose functions are $0,1,+$ and $kx$ (which we may also denote $k\cdot x$).
We will write $\mathbb M_k$ instead of $\mathbb T^\mathcal M_k$.
For ease of notation, we will omit parentheses around addition and treat terms $(\sigma+\tau)+\rho$ and $\tau+(\sigma+\rho)$ as identical; this will not be an issue, as all of the properties we consider are invariant under associativity.
For $q\in \mathbb N$, we define a term $\bar q$ by setting $\bar 0 = 0$ and, for $q>0$, $\bar q= 1+ 1 + \cdots +1$ ($q$ times); note that $\norm{\bar q} = 2q-1$ in this case.
For $k\geq 2$, define $m\gknf k\cdot p + q$ if $p,q$ are the unique positive integers such that $q<k$ and $m = k\cdot p + q$.
If $p=0$, set $\nf km =  \bar q$, and if $p>0$, define inductively $\nf k m = k \cdot \nf k p + \bar q$.
Note that $\norm{m }_k = \norm{p }_k + 2 q+1$.
Throughout this section, all notation (e.g.~$\nf km$, $\norm \tau$, etc.) will refer exclusively to this representation of natural numbers.

\begin{lemma}\label{lemmMultBCH}
If $m\in \mathbb N$ and $\ell > k\geq 2$, then
\[\nf{\ell }{\bch m {\bases k\ell}} = \bch {\nf k m}{^\ell}.\]
\end{lemma}

\proof
This is clear since $q<k$ yields $q<\ell$.
\endproof

%\begin{theorem}\label{theoMultBchMon}
%If $m<n$ then $\bch m{k+1}<\bch n{k+1}$.  (Property C2)
%\end{theorem}

%\proof
%The claim should be clear if $m=0$ so we assume otherwise.
%Assume that $m \gknf k \cdot p+q$ and $n \gknf k\cdot p'+q'$.
%If $p<p'$, the induction hypothesis yields $\bch p{k+1}<\bch {p'}{k+1}$.
%Hence
%\begin{align*}
%\bch m{k+1} & = (k+1)\cdot \bch p {k+1}+q<(k+1)\cdot \bch{ p'} {k+1}\\
%&\leq (k+1)\cdot \bch{ p'} {k+1} +q'=\bch{ n} {k+1}.
%\end{align*}
%Otherwise, $p=p'$ and $q<q'$.
%Then 
%\begin{align*}
%\bch m{k+1}=(k+1)\cdot \bch p {k+1}+q<(k+1)\cdot \bch {p'}{k+1}+q'=\bch {n}{k+1}.
%\end{align*}
%\endproof

This normalized notation system satisfies both optimality properties, as we see next.

\subsection{Norm Minimality}

Let us begin by showing that our multiplicative notation system satisfies the norm minimality property.

\begin{theorem}
If $\tau \in \mathbb M_k$, then $\norm {\nf k \tau} \leq \norm \tau$.
\end{theorem}
\shortproof{
\proof
Write $m = \val \tau$.
The claim is proven by induction on $m$, considering several cases.
We treat the most interesting case, which is that of a term 
$\tau = k\cdot \sigma + \bar n$.
Write $n = k\cdot p + q$ with $q<k$, so that $m\gknf k\cdot \val { \sigma   + \bar p} + q$.
By the induction hypothesis, $\norm {\val{ \sigma + \bar p}}_k \leq \norm { \sigma + \bar p} =\norm   \sigma + 2 p  $.
Hence,
\begin{align*}
\norm m_k &\leq \norm {  \val{ \sigma+p}}_k + 2 q +1  = \norm   \sigma + 2 (p + q) + 1 \\
& \leq \norm \sigma + 2 n + 1   = \norm {k\cdot \sigma + \bar n}.
\end{align*}
\endproof
}

\longproof{
\proof
Write $m = \val \tau$ and proceed by induction on $m$, considering several cases.
\medskip

\begin{enumerate}[label*={\sc Case \arabic*},wide, labelwidth=!, labelindent=0pt]

\item ($\tau = 0$). Then, $m = 0 $ and $\norm m_k = \norm \tau = 1$.
\medskip

\item ($\tau = k\sigma$). If $\val \sigma = 0$ then $\nf k\tau = 0$ and $\norm 0 <\norm \tau$. Otherwise, by the induction hypothesis, $\norm{\tau} \geq \norm {k \cdot \nf k\sigma}$, and clearly $k\cdot \nf k\sigma = \nf k\tau$.
\medskip

\item ($\tau = k\cdot \sigma + \bar n$).
We may assume that $n>0$, otherwise $m = \val{k\cdot \sigma}$ with smaller norm, and we can apply the previous case.
Write $n = k\cdot p + q$ with $q<k$, so that $m\gknf k\cdot \val { \sigma   + \bar p} + q$.
By the induction hypothesis, $\norm {\val{ \sigma + \bar p}}_k \leq \norm { \sigma + \bar p} =\norm   \sigma + 2 p  $.
Hence,
\begin{align*}
\norm m_k &\leq \norm {  \val{ \sigma+p}}_k + 2 q +1  \leq \norm   \sigma + 2 (p + q) + 1 \\
& \leq \norm \sigma + 2 n + 1   = \norm {k\cdot \sigma + \bar n}.
\end{align*}

\item ($\tau = \bar m $). If $ m < k$, then $ \tau$ is already in normal form.
Otherwise, we note that $m = \val{ k\cdot 1 + \overline{m-k} }$ (omitting $\overline{m-k}$ if $m=k$) and, moreover, $\norm{k\cdot 1 + \overline{m-k}} = 2 + 2(m-k) \leq 2 + 2(m-2) < 2m-1 $.
Thus we obtain $\norm \tau > \norm {k\cdot 1 + \overline{m-k}}$, and can apply one of the previous cases.
\medskip

\item ($\tau = \tau_0 + \tau_1$, but not one of the above).
By the induction hypothesis, we have that $\norm \tau \leq \norm {( \nf k{  \tau_0} + \nf k{  \tau_1})} $, and thus we may assume that $\tau_0$ and $\tau_1$ are in normal form.
Write $ {\tau_0}=k\sigma_0 + \bar  n_0 $ and $ {\tau_1} = k\sigma_1 + \bar n_1 $, where any of the displayed terms are omitted if their value is zero.
Then, $\val \tau = \val {k(\sigma_0 + \sigma_1) + \overline {(n_0 + n_1)}} $ (also omitting null terms), and it is easy to check that $ \norm{\tau} \geq \norm {k(\sigma_0 + \sigma_1) + \overline {(n_0 + n_1)}}$.
We can then apply the appropriate one of the previous cases.
\end{enumerate}
\endproof
}

\subsection{Maximality of Base Change}

Recall that our second optimality criterion was optimality under base change.
We will show that multiplicative normal forms also enjoy this property.
This will follow from the next lemma.

\begin{lemma}\label{lemmMultMax}
If $m = k\cdot r+s$ and $\ell \geq k$, then
\[
\bch {m}{\bases k\ell}\geq \ell \cdot \bch {r}{\bases k\ell}+s,\]
and equality holds if and only if $m \gknf k\cdot r+s$.
\end{lemma}

\proof
In this proof, we write $\bch x{} $ instead of $\bch x{\bases k\ell}$.
Proceed by induction on $m$.
If $m=0$, then $r=s=0$ and $\bch 0 { }= 0 = \ell \cdot \bch {0}{ }+0$, so we assume $m>0$.
Write $s\nnf k\cdot p + q$ (with $p$ or $q$ possibly zero), so that $m\nnf k\cdot(r+p) +q$.
Write $r = ku+v$ in normal form.
Then, the induction hypothesis yields
\[ \bcn {(r+p)} {} = \bch {\big (ku+( v+p) \big )} {} \stackrel{\text{\sc ih}}\geq \ell \cdot \bcn{u} {^\ell} + v + p = \bcn r{^\ell} + p. \]
Hence,
\begin{align*}
\bch m { } & = \ell  \cdot \bch {(r+p)}{} +q
 \geq \ell \cdot \bch {r} {} + \ell  p+q \geq \ell  \cdot \bch {r} {} +s,
\end{align*}
and the last inequality is strict unless $p = 0$, so that $m \gknf kr+s$.
\endproof

In view of Lemma~\ref{lemmMultMax}, a simple induction on term complexity yields the following.

 \begin{theorem}\label{theoMultMax}
If $\tau \in \mathbb M_k$, $\ell > k \geq 2$, and $m=\val \tau$, then
\[\bch { m}{\bases k\ell } \geq \val{\bch \tau{^\ell}}.\] 
 \end{theorem}

\section{Optimality of Exponential Normal Forms}\label{secExponential}

Now we turn our attention to the original Goodstein process.
In this setting, it is already known that the process terminates \cite{Goodsteinb}, and that this fact is independent of Peano arithmetic \cite{Kirby}.
We will show that the notation system used satisfies our optimality criteria.
We begin by establishing some useful basic properties.

\subsection{Properties of Normal Forms}

Recall that $\mathcal E$ has as primitive functions $0$, $x+y$, and $ k^x$ (with $k\geq 2$), and is equipped with normal forms as defined in Section \ref{secClassical}. We will treat terms modulo associativity of addition and hence omit parentheses.
However, we will not treat term addition as commutative.
With this in mind, it is easy to check that $k^{\rho_0} + \ldots + k^{\rho_{n-1}}$ is in normal form if and only if each $\rho_i$ is in normal form, $\rho_i \geq \rho_{i+1}$ whenever $i + 1  < n $, and $\rho_i > \rho_{i+k-1}$ whenever $i+k-1 < n$.
We will extend the notation $\nnf$ to write $m\nnf \tau(k,a_1,\ldots,a_{n })$, where $a_i\in \mathbb N$, if $m\gknf \tau(k,\nf k{a_1},\ldots,\nf k {a_{n }})$; for example, we may write $15=_2 2^3+ 2^2 + 3$ or $12=_2 8 + 2^2$ but not, say, $15=_2 7 + 2^3$.
Sums should be read from right to left, i.e.~$\sum_{i=0}^{n } \tau_i = \tau_n + \ldots + \tau_0$.
Multiplication is used as a shorthand: $p\cdot \tau = \tau+\ldots + \tau$ ($p$ times).

With this notation at hand, the following is easily checked.

\begin{lemma}\label{lemmExpNFProp}
Fix $k\geq 2$, $m\in \mathbb N$ and $ \sigma,\tau \in \mathbb E_k$.
\begin{enumerate}

\item If $\sigma + \tau $ is in normal form, then $\sigma$ and $\tau$ are each in normal form.

\item\label{itExpNFPRopOne} If $m \nnf k^a$ and $b < a$, then $m-k^b \nnf \sum_{i=  b}^{a-1} (k-1) k ^{i}$.

\item\label{itExpNFPRopTwo} If $m = a + k^b$ and $n = k^c + d$ are in normal form with $b>c$, then $m+n$ is in normal form.

\end{enumerate}
\end{lemma}

\subsection{Norm Minimality}

In this subsection, we will show that the hereditary exponential notation satisfies norm minimality.
We begin with some useful inequalities.

\begin{lemma}\label{lemmExpPlusOne}
If $k\geq 2$ and $m\in \mathbb N$, then $\norm{m+1}_k \leq   \norm m_k+3$.
\end{lemma}

\proof
Write $m+1 \nnf a+k^b$.
If $b = 0$, then $ m =  a$ and $\norm{m+k^0}_k = \norm m_k + 3$ (we add one for the term $0$, one for $+$ and one for $k^\cdot$).
Otherwise, using Lemma \ref{lemmExpNFProp}, we see that $m \nnf a + \sum_{i=0}^{b-1} (k-1) k^i$, and
\begin{align*}
\norm m_k + 3 & \geq \norm {a + k^{b-1}}_k + 3 = \norm a_k +\norm{b-1}_k +5 \\
& \geq^{\text{\sc ih}} \norm a + \norm{b} + 2 = \norm {m+1}_k.&
\end{align*}
\endproof

\shortproof{With this in mind, the following useful inequality is proven by induction on $m$; details are left to the reader.}

\begin{lemma}\label{lemmExpNormMin}
If $m= k^a + b$, then $\norm m_k \leq \norm a_k + \norm b_k + 2$.
\end{lemma}

\longproof{
\proof
Induction on $m$. Write $m \nnf k^p + q$.
Consider the following cases.
\begin{enumerate}[label*={\sc Case \arabic*},wide, labelwidth=!, labelindent=0pt]

\item\label{itExpMinOne} ($ a = p$). Then also $q = b$, so that
\[  \norm a_k + \norm b_k + 2  =  \norm p_k + \norm q_k + 2 = \norm m_k. \]

\item ($  b \geq k^{p}$). Then, $b \nnf k^p + c$ for $ c = q- k^a $, and the induction hypothesis yields
\begin{align*}
\norm a_k + \norm b_k + 2 & = \norm a_k + \norm p_k + \norm c_k + 4 \\
&\geq^{\text{\sc ih}} \norm p_k + \norm{k^a + c}_k + 2 = \norm p_k + \norm{q}_k + 2 = \norm m_k.
\end{align*}

\item ($ a <  p$ and $  b <  k ^p $).
From $k^p>b = k^p + q - k^a$ we obtain $q < k^a $.
Thus by Lemma \ref{lemmExpNFProp},
\[b = ( k^p - k^a ) +  q \nnf \sum _{ i = a }^{ p-1 } (k-1) k^i + q. \]
Hence,
\[
\norm b_k \geq \norm{k^{p-1} + q}_k = \norm {p - 1}_k +  \norm q_k + 2
\geq \norm { p }_k + \norm q_k - 1,
\]
where the last inequality is by Lemma \ref{lemmExpPlusOne},
so that
\[
\norm a_k + \norm b_k + 2 \geq \norm a_k + \norm { p }_k + \norm q_k + 1 \geq \norm m_k.
\]
\end{enumerate}
\endproof
}

In particular, if $m\nnf k^p+q$ and $k^a+b=m$, then we have that $ \norm p_k +\norm q_k+2=  \norm m_k\leq \norm a_k+\norm b_k +2 $.
From this and an easy induction on term complexity, we obtain that $\mathcal E$ is norm minimal.

\begin{theorem}\label{theoExpNormMin}
If $\tau \in \mathbb E_k$ and $m=\val \tau$, then $\norm m_k \leq \norm \tau$.
\end{theorem}

\subsection{Maximality of Base Change}

We have seen that hereditary exponential notation satisfies norm minimality.
Let us now show that it is base-change maximal as well.
This will follow from the next lemma.
If $\mathcal F$ is a normalized notation system, say that $\mathcal F$ is base-change maximal below $m\in \mathbb N$ if, whenever $\tau \in \mathbb T^\mathcal F_k$ and $\val \tau < m$, it follows that $ \val{\bch \tau {^\ell} } \leq \val{ \bch{\nf{k}\tau}{^\ell}} $.
Recall from Remark \ref{remMaxToMon} that, if $\mathcal F$ is base-change maximal below $m$, then whenever $x<y<m$, we may conclude that $\bch x{\bases k\ell} < \bch y{\bases k\ell}$.
As we wish to appeal to this property in the proof of the following lemma, we will assume inductively that hereditary exponential notation is base-change maximal below $m$.

\begin{lemma}\label{lemmExpMaxBCH}
Fix $\ell > k \geq 2$ and write $\bch x{}$ instead of $\bch x {\bases k\ell}$.
Suppose that the normalized notation system $\mathcal E$ is base-change maximal below $m$.
If $m=k^a+b$, then
\[\bcn m {\bases k\ell} \geq \ell^{\bcn a{\bases k\ell}} + \bcn b{\bases k\ell}.\]
\end{lemma}

\shortproof{
\proof
Write $m \gknf k^p + q$.
The proof proceeds by induction on $m$; here we treat the critical case, where $ a <  p$ and $  b <  k ^p$.
We note that in this case $q < k^a$, and thus
\[b = ( k^p - k^a ) + q \nnf  \sum _{ i = a }^{p-1} (k-1) k^{i} + q, \]
hence
\begin{align*}
\bch b{}  = (k-1) \sum _{ i = a }^{p-1}  \ell^{\bch {i}{}} + \bch q{}  = (k-1) \sum _{ i = 1 }^{p-a}  \ell^{\bch {(p-i)}{}} + \bch q{} .
\end{align*}
Since $p<m$, we may use the assumption that $\mathcal E$ is base-change maximal below $m$ to obtain $\bch{(p-i)} {} \leq \bch{ p} {} -i$, and hence
\[  \sum _{ i = 1 }^{p - a}  \ell^{\bch {(p - i)}{}} \leq \sum _{ i = 1 }^{p - a }  \ell^{\bch {p }{} - i} = \dfrac{\ell^{\bch p{}}(\ell^{p - a } - 1)}{\ell^{p - a  }(\ell-1)} < \dfrac{\ell^{\bch p{}} }{ \ell-1 } \leq \frac { \ell^{\bch p{} }}{k },\]
so $\bch b{} < \nicefrac { (k-1)\ell^{\bch p{} }}{k } + \bch q{}$.

By monotonicity below $m$, available due to Remark~\ref{remMaxToMon}, we have that $\bcn a{} \leq \bcn p{} -1 $, so $\ell^{\bch a{}} \leq \nicefrac{\ell^{\bch p{}}}\ell < \nicefrac{\ell^{\bch p{}}} k$. Therefore,
\[
\ell^{\bch a{}} + \bch b{} < \frac{\ell^{\bch p{}}} k + \frac {(k-1) \ell^{\bch p{} }}{k } + \bch q{} =   \ell^{\bch p{}}  + \bch q{}= \bch m{}.
\]
\endproof
}

\longproof{
\proof
By induction on $m$.
Write $m \gknf k^p + q$ and consider the following cases.
\medskip
\begin{enumerate}[label*={\sc Case \arabic*},wide, labelwidth=!, labelindent=0pt]

\item\label{itExpMaxOne} ($ a = p $). Then also $  b=q$ and $\bch m{}   = \ell^{\bch a{}} + \bch b{}.$
\medskip

\item ($a < p$ and $  b \geq k ^{p}$). Then, $b \gknf k^p + c$ for some $c$, and the induction hypothesis yields
\begin{align*}
\ell^{\bch a{}} + \bch b{}  & \leq \ell ^{\bch a{}} + \ell ^{\bch p{}} + \bch c{}
 =  \ell^{\bch p{}} + \ell^{\bch a{}} + \bch c{} \\
&  \leq  \ell^{\bch p{}} + \bch{( k^{a} +  c ) }{} =   \ell^{\bch p{}} + \bch{q }{} = \bcn m{}.
\end{align*}
\medskip

\item ($ a <  p$ and $  b <  k ^p$). As in the proof of Lemma \ref{lemmExpNormMin}, $q < k^a$, and thus
\[b = ( k^p - k^a ) + q \nnf  \sum _{ i = a }^{p-1} (k-1) k^{i} + q, \]
hence
\begin{align*}
\bch b{}  = (k-1) \sum _{ i = a }^{p-1}  \ell^{\bch {i}{}} + \bch q{}  = (k-1) \sum _{ i = 1 }^{p-a}  \ell^{\bch {(p-i)}{}} + \bch q{} .
\end{align*}
Since $p<m$, we may use the assumption that $\mathcal E$ is base-change maximal below $m$ to obtain $\bch{(p-i)} {} \leq \bch{ p} {} -i$, and hence
\[  \sum _{ i = 1 }^{p - a}  \ell^{\bch {(p - i)}{}} \leq \sum _{ i = 1 }^{p - a }  \ell^{\bch {p }{} - i} = \dfrac{\ell^{\bch p{}}(\ell^{p - a } - 1)}{\ell^{p - a  }(\ell-1)} < \dfrac{\ell^{\bch p{}} }{ \ell-1 } \leq \frac { \ell^{\bch p{} }}{k }.\]
\end{enumerate}
By monotonicity below $m$, we have that $\bcn a{} \leq \bcn p{} -1 $, so $\ell^{\bch a{}} \leq \nicefrac{\ell^{\bch p{}}}\ell < \nicefrac{\ell^{\bch p{}}} k$. Therefore,
\[
\ell^{\bch a{}} + \bch b{} < \nicefrac{\ell^{\bch p{}}} k + \nicefrac {(k-1) \ell^{\bch p{} }}{k } + \bch q{} = \bch m{}.
\]
\endproof
}

\begin{theorem}\label{theoExpBCHMax}
If $2\leq k<\ell$, $\tau \in \mathbb E_k$, and $m=\val \tau$, then $\bch m {^{\ell}_k} \geq \val{\bch \tau {^{\ell}}}$.
\end{theorem}

\proof
Induction on term complexity using Lemma \ref{lemmExpMaxBCH}.
\endproof

In view of Proposition \ref{propMaxToMon}, we immediately obtain monotonicity of the base-change operation.

\begin{corollary}\label{corExpMon}
If $m<n$ and $2\leq k < \ell$, then $\bch m{\bases k\ell} < \bch n{\bases k\ell}$.
\end{corollary}

From monotonicity, we readily obtain normal form preservation for hereditary exponential normal forms.

\begin{lemma}\label{lemmExpPreserve}
If $m\in \mathbb N$ and $\ell > k\geq 2$, then
\[\nf{\ell }{\bch m {\bases k\ell}} = \bch {\nf k m}{^\ell}.\]
\end{lemma}

\begin{proof}
Write $\bch{}{}$ for $\bch{}{\bases k\ell}$.
It suffices to show that if $k^a+b$ is in normal form then so is $\ell^{\bcn a {\bases k\ell} } + {\bcn b {\bases k\ell} }$, since then the result follows by an easy induction on $\|\nf k m \|$.

We write $k^a+b = r\cdot k^a+c $, where $c<k^a$ and $0<r<k$, and proceed by a secondary induction on $s\leq r $ to prove that $ (s+1)\cdot \ell^{\mchb {}{} a }  > s\ell^{\mchb {}{} a  }+ \mchb {}{} c$.
The claim will then follow, since $b\gknf (r-1) k^{ {} {}{a}  }+   {} {}{c} $ and
\begin{align*}
\ell^{\mchb {}{} a +1} & =  \ell \cdot \ell^{\mchb {}{} a }   \geq (r+1)\cdot \ell^{\mchb {}{} a }\\
& > \ell^{\mchb {} {}{a}  }+ (r-1) \ell^{\mchb {} {}{a}  }+ \mchb {} {}{c} = \ell^{\mchb {} {}{a}  }+ \mchb {} {}{b},
\end{align*}
as needed.

For $s=1$,  $c <k^a$ and Corollary~\ref{corExpMon} yields $\mchb {}{} c< \mchb {}{} k^a = \ell^{\mchb {}{} a} $
Hence, $(1+1)\ell^{\mchb {}{} a} >  \ell^{\mchb {}{} a}+ \mchb {}{} c$.
Otherwise,
\[(s+1)\cdot \ell^{\mchb {}{} a  }  = \ell^{\mchb {}{} a  } + s\cdot \ell^{\mchb {}{} a  }    \stackrel{\text{\sc ih}}> \ell^{\mchb {}{} a  } + (s-1) \ell^{\mchb {}{} a  }  + \mchb {}{} c = s \ell^{\mchb {}{} a  }  + \mchb {}{} c. \]
\end{proof}

\done

\section{Elementary functions}\label{secElementary}

In this section we consider an extension of $\mathcal E$ with product and study whether hereditary exponential normal forms are still optimal in this context.
Define $\mathcal L = \{0,x+y,x\cdot y,k^x\}$.
Then for example,
\begin{align*}
(5^2+5^1+5^0)\cdot ( 5^1+5^0) & = 5^3+2\cdot 5^2+2\cdot 5^1 + 5^0 \\
&= 5^3+ 5^2+5^2+5^1 + 5^1 + 5^0,
\end{align*}
although the left hand side has the smallest norm of the three.
This tells us that exponential normal forms no longer give minimal norms, even if we allow for coefficients below $k$.
However, as we will see, we still obtain maximality under base change. Below and throughout this section, $\nf\cdot \cdot $ denotes the normal form operator for $\mathcal E$, as used in the standard Goodstein theorem.

\begin{theorem}\label{theoElemMax}
Let $\ell \geq k\geq 2$,  $m\in \mathbb N$, and $\tau \in \mathcal L_k $. Then, $\val{\bch{\tau}{^\ell}} \leq \val{\bch { \nf k \tau} {^\ell}}$.
\end{theorem}

\shortproof{

\begin{proof}
The theorem is proven by induction on $\val \tau$ with a secondary induction on $\norm \tau$.
Write $\bch{}{}$ for $\bch{}{^\ell}$.
The key step is reducing a product to a term in $\mathbb E_k$.
Suppose that $\tau = \sigma \cdot \rho$, with both terms having non-zero value.
Write $\nf{k}\sigma = k^\alpha + \beta $ and $\nf{k}\rho = k^\gamma+ \delta $.
Define $\tau'= k^\alpha \cdot \delta +  k^\gamma \cdot \beta + \beta\cdot \delta$.
Then,
\begin{align*}
\bcn \tau \ell & = \bcn \sigma \ell\cdot  \bcn \rho \ell  \leq ^{\text{\sc ih}} \bcn {(k^\alpha + \beta )} \ell \cdot \bcn {(k^\gamma+ \delta)} \ell \\
&= \ell^{\bcn  \alpha \ell + \bcn \gamma\ell} + \bcn{\tau'}\ell \leq^{\text{\sc ih}} \ell^{\bcn  \alpha \ell + \bcn \gamma\ell} + \bcn{\nf{k}{\tau'}}\ell  = \bcn{(k^{\alpha+\gamma} + \nf{k}{\tau'} )}\ell,
\end{align*}
where the first inductive step is by the secondary induction hypothesis on $\norm{\sigma},\norm{\rho}<\norm{\tau}$ and the second is the primary induction hypothesis on $|\tau'|<|\tau|$.
Note that $k^{\alpha+\gamma} + \nf{k}{\tau'} \in \mathbb E_k$, and moreover
\[\val\tau = \val{k^{\alpha+\gamma} + {\tau'}} = \val{k^{\alpha+\gamma} + \nf{k}{\tau'}}.\]
Thus by Theorem \ref{theoExpBCHMax},
\[\val{\bcn{(k^{\alpha+\gamma} + \nf{k}{\tau'} )}\ell} \leq \val{ \bcn{\nf{k}{k^{\alpha+\gamma} + \nf{k}{\tau'} }}\ell}
= \val{\bcn {\nf{k} \tau}\ell }.\]
We conclude that $\val{\bcn \tau \ell} \leq \val{ \bcn {\nf{k}\tau}\ell }$, as required.
\end{proof}
}

\longproof{

\begin{proof}
By induction on $\val \tau$ with a secondary induction on $\norm \tau$.
Write $\bch{}{}$ for $\bch{}{^\ell}$.
\medskip

\begin{enumerate}[label*={\sc Case \arabic*},wide, labelwidth=!, labelindent=0pt]

\item ($\tau = 0$). Trivial.
\medskip

\item ($\tau = \sigma + \rho$). Let $ \sigma' = \nf k{ \sigma} $ and $\rho' = \nf k{ \rho} $.
Note that $\val{\sigma},\val{\rho} \leq \val \tau$ and $\norm {\sigma}, \norm{\rho } <\norm {\tau}$.
Thus we may use the secondary induction hypothesis to see that
\[\bch {\tau}{} = \bcn{ ( \sigma+\rho ) }{} =\bcn \sigma{} +\bcn \rho {} \leq \bcn{   \sigma'}{}+ \bcn{\rho'  } {} = \bcn{ (  \sigma' + \rho' ) } {}  \leq \bcn{\nf {k}{ \sigma' + \rho'}}{},\]
where the last inequality uses the fact that $ { { \sigma' + \rho'}}{}\in\mathbb E_k$ and Theorem \ref{theoExpBCHMax}.
Since $\nf {k}{ \sigma' + \rho'} = \nf {k}{ |\tau|} = \nf {k}{ \tau}$, the required inequality follows.

\item ($\tau = k^\sigma$). Then,
$\bcn\tau\ell=\ell^{\bcn \sigma \ell}  \leq^{\text{\sc ih}} \ell^{\bcn{\nf{k}\sigma} \ell} =\bcn{\nf{k}\tau} \ell . $

\item ($\tau = \sigma \cdot \rho$).
If $\val \sigma = 0$, an easy induction shows that $\bcn \sigma{} = 0$, since $\sigma$ is either $0$, a sum of terms with value $0$, or a product of terms one of which has value $0$.
From $\norm \sigma < \norm \tau$ and the secondary induction hypothesis we have that
\[\bcn \tau \ell = \bcn \sigma \ell \cdot \bcn \rho \ell \leq \bcn 0{\bases k\ell}\cdot \bcn \rho \ell = 0 =\bcn {\nf k\tau}\ell ,\]
so we may assume that $\val \sigma >0$ and, by the same reasoning, that $\val \rho > 0$.
Write $\nf{k}\sigma = k^\alpha + \beta $ and $\nf{k}\rho = k^\gamma+ \delta $.
Define $\tau'= k^\alpha \cdot \delta +  k^\gamma \cdot \beta + \beta\cdot \delta$.
Then,
\begin{align*}
\bcn \tau \ell & = \bcn \sigma \ell\cdot  \bcn \rho \ell  \leq ^{\text{\sc ih}} \bcn {(k^\alpha + \beta )} \ell \cdot \bcn {(k^\gamma+ \delta)} \ell \\
&= \ell^{\bcn  \alpha \ell + \bcn \gamma\ell} + \bcn{\tau'}\ell \leq^{\text{\sc ih}} \ell^{\bcn  \alpha \ell + \bcn \gamma\ell} + \bcn{\nf{k}{\tau'}}\ell  = \bcn{(k^{\alpha+\gamma} + \nf{k}{\tau'} )}\ell,
\end{align*}
where the first induction is by the secondary induction hypothesis on $\norm{\sigma},\norm{\rho}<\norm{\tau}$ and the second is the primary induction hypothesis on $|\tau'|<|\tau|$.
Note that $k^{\alpha+\gamma} + \nf{k}{\tau'} \in \mathbb E_k$, and moreover
\[\val\tau = \val{k^{\alpha+\gamma} + {\tau'}} = \val{k^{\alpha+\gamma} + \nf{k}{\tau'}}.\]
Thus by Theorem \ref{theoExpBCHMax},
\[\val{\bcn{(k^{\alpha+\gamma} + \nf{k}{\tau'} )}\ell} \leq \val{ \bcn{\nf{k}{k^{\alpha+\gamma} + \nf{k}{\tau'} }}\ell}
= \val{\bcn {\nf{k} \tau}\ell }.\]
We conclude that $\val{\bcn \tau \ell} \leq \val{ \bcn {\nf{k}\tau}\ell }$, as required.
\end{enumerate}
\end{proof}
}

Theorem \ref{theoElemMax} might seem surprising, as hereditary exponential normal forms do not involve multiplication, yet they remain base-change maximal even compared to arbitrary elementary terms. Later, we will see that this result leads to a wide generalization of Goodstein's principle.

\section{Termination times of the Goodstein processes}\label{secLowerBound}

In this section we provide a proof that the Goodstein processes terminate by comparing them to Hardy functions.
These are functions defined by transfinite induction up to $\varepsilon_0$.
Our analysis will yield additional information which will also lead to independence results.
We begin by reviewing these functions and their properties.

%%%%BOOKMARK

\subsection{Hardy functions and independence}

Recall that $\varepsilon_0$ is the first fixed point of the function $\xi\mapsto\omega^\xi$, and by the Cantor normal form theorem, every non-zero $\xi<\ve_0$ can be written in the form $\omega^\alpha+\beta$ with $\alpha,\beta<\xi < \omega^{\alpha+1}$.
The Cantor normal form of $0$ is $0$.

\begin{definition}\label{defFS}
For $\xi<\ve_0$ and $n \in \mathbb N$, we define $\xi[n]$ recursively by
\begin{itemize}
    \item $0[ n ]  :=   1 [ n ] :=0 $;

    \item $(\omega^\alpha+\beta)[ n ] := \omega^\alpha+\beta[n] $ if $\omega^\alpha+\beta$ is in Cantor normal form and $\be>0$;

    \item $(\omega^{\alpha+1}) [ n ] := \omega^\alpha  n $;
    
        \item $(\omega^{\alpha}) [ n ] := \omega^{\alpha[n]}$ if $\alpha$ is a limit.
\end{itemize}
\end{definition}

These fundamental sequences satisfy the essential properties that $\fs \al n<\al$ if $\al\neq 0$ and, if $\al<\ve_0$ is a limit ordinal, then $(\al[n])_{n<\om}$ is an increasing sequence converging to $\al$.
Another key property of these fundamental sequences is the {\em Bachmann property;} see \cite{Schmidt77,WeiermannBSL} for details.

\begin{proposition}[Bachmann property]\label{propBachmann}
If $\alpha,\beta < \ve_0$ and $k \in\mathbb N$ satisfy $\fs \al k < \beta < \al$, then $\fs\al k \leq \fs \be 1$.
\end{proposition}

For example, $\fs {\om^\om} 4=\om^ 4$ and $\om ^4<\om^ 6<\om^\om$, while $\om ^4< \om^ 5 =\fs{ \om^6  }1 $.
The intuition here is that if $\be\in (\fs \al k,\al)$ and we iteratively apply fundamental sequences to $\be$, we will be `stuck' in the interval $(\fs \al k,\al)$, unless we pass through $\fs \al k$.
Note that this may fail if we replace $1$ by $0$, as $\fs{\om^6}0 = 0< \fs {\om^\om} 4$.

We also need to identify conditions under which we can guarantee that $\xi[n]\geq \zeta$ for $\zeta<\xi$.
To this end, define the {\em maximal coefficient} of $\zeta$, $\mc\zeta$, to be the largest natural number appearing in $\zeta$ when written in Cantor normal form.
To be precise, $\mc 0=0$, and if $\zeta = \omega^\alpha n+\beta$ with $\beta<\omega^\alpha$, then $\mc\zeta = \max\{n,\mc\alpha,\mc\beta\}$.
The following is proven in e.g.~\cite{FWTheta}.

\begin{lemma}\label{lemmFSN}
Let $\xi<\ve_0$ and $n\in\mathbb N$.
Then,
\begin{enumerate}

\item $\mc{\fs \xi n } \leq \max \{\mc \xi, n\}$, and

\item if $\zeta<\xi$ and $\mc\zeta< n$ then $\zeta\leq \xi[n]$.

\end{enumerate}
\end{lemma}

Fundamental sequences can be used to define fast-growing functions on the natural numbers, such as the {\em Hardy functions} below.

\begin{definition}
    For $x\in \mathbb{N}$ and $\alpha  < \varepsilon_0$, we define
    \begin{itemize}
    \item $H_0 (x) = x$;
     \item $H_\al (x) = H_{\al[x]}(x + 1)$ if $\al\neq 0$.
\end{itemize}

\end{definition}

The intuition is that each $H_\alpha$ is an increasing function, which grows more quickly for larger $\alpha$.

\begin{theorem}[\cite{BCW}]\label{theoOrdMon}\

\begin{enumerate}

\item If $x<y$ and $\al<\ve_0$, then $H _\al(x) < H _\al(y)$.

\item If $\al<\be$ and $x > \mc\al$, then $H _\al(x) < H _\be(x)$.

\end{enumerate}
\end{theorem}

The totality of these functions cannot be proven over weak theories.
To make this precise, define $\omega_n$ to be a tower of $n$ $\omega$'s, i.e.~$\omega_0 = 1 $ and $\omega_{n+1} = \omega^{\omega_n}$.

\begin{theorem}[\cite{FairtloughWainer,BuchholzTotal}]\label{theoHIndep}
For $n\in\mathbb N \setminus\{0\}$,
\begin{enumerate}

\item 
${\sf I}\Sigma_n$ proves that $H_\alpha $ is total if and only if $\alpha< \omega_{n+1}$.

\item
If $f$ is a provably total computable function in ${\sf I}\Sigma_n$, then there is $N\in \mathbb N$ such that for all $x>N$, $f(x) < H_{ \omega_{n+1}}(x)$. 

\end{enumerate}
\end{theorem}

Thus a general strategy for proving independence of $\Pi^0_2$ statements is showing that they require witnesses growing faster than suitable Hardy functions.
This approach based on the Hardy function and variants has been used in various classic independence results by e.g.~Cichon~\cite{CichonIndependence}, Loebl and Nešetřil~\cite{Loebl}, and, of course, Kriby and Paris' proof of independence of Goodstein's theorem~\cite{Kirby}.
The following is useful in establishing such lower bounds.

\begin{proposition}\label{propHLower}
Let $x> 0$, $\xi<\varepsilon_0$, and $(\xi_n)_{n \leq \ell}$ be a sequence of ordinals below $\ve_0$ such that
\begin{enumerate}

\item $\xi[x]< \xi_0 \leq \xi  $,

\item for all $n < \ell$, $ {\xi_n}[x+n] \leq \xi_{n+1} \leq \xi_n$, and

\item 

 $\xi_\ell = 0$.

\end{enumerate}

Then, $\ell + x \geq H_\xi(x)$.
\end{proposition}

\begin{proof}
 We prove the lemma by induction on $\ell$.
 The claim is vacuously true when $\ell = 0$ (i.e., $\xi_0 = 0$), so we assume otherwise.
 Consider two cases.
 \begin{Cases}

\item ($\xi[x] <\xi_1  $).
Then, the sequence $(\xi_{n+1})_{n \leq \ell-1}$ once again satisfies the assumptions, since ${\xi_{n+1}}[x+n] \leq {\xi_{n+1}}[x+n+1] \leq \xi_{n+2}  $.
By induction on $\ell$, we obtain $\ell - 1 + x \geq H_\xi(x)$.
\smallskip

\item ($\xi[x] = \xi_1  $).
Then, the sequence $(\xi_{n+1})_{n \leq \ell-1}$ satisfies the assumptions, but for $\xi$ replaced by $\xi[x]$ and $x$ replaced by $x+1$, since ${\xi_{n+1}}[(x+1)+n] = {\xi_{n+1}}[x+n+1] \leq \xi_{n+2}  $.
Thus we obtain $\ell+x = (\ell - 1)+(x+1) \geq H_{\xi[x]}(x+1) = H_\xi(x)$.
 \end{Cases}
\end{proof}

We will also need to establish {\em upper} bounds in terms of the Hardy functions.
For this, we use the following.

\begin{proposition}\label{propHUpper}
Let $x\in\mathbb N$ and $(\xi_n)_{n \leq \ell}$ be a sequence of ordinals below $\ve_0$ such that $ \xi_0 \leq \xi  $, $\xi_{n}>0$ if $n<\ell$, and for all
$n < \ell$, $ \xi_{n+1} < \xi_n$ and $\mc{\xi_{n+1}} \leq x+n $.
Then, $\ell + x < H_\xi(x+1)$.
\end{proposition}

\begin{proof}
If $\xi_0=0$, then $\ell+x = x <x+1 = H_0(x+1) $.
Otherwise, consider the sequence $(\xi_{n+1})_{n \leq \ell-1}$; we claim that it satisfies the assumptions of the proposition, but with $\xi$ replaced by $\xi[x+1]$ and $x$ replaced by $x+1$.
By Lemma \ref{lemmFSN}, $\xi_1\leq \xi [x+1] $, and $\mc{\xi_{n+2}} \leq x+n+1 = (x+1)+n $.
So, induction on $\ell$ yields $(\ell-1)+(x+1) < H_{\xi[x+1]}( x+ 2) =  H_\xi(x+1) $, i.e.~$\ell +x < H_\xi(x+1)$, as needed.
\end{proof}

\subsection{Termination of the weak Goodstein process}

We now show that the weak Goodstein process terminates in time $H_{\omega^\omega}$.
In this subsection, all notation (e.g.~$\nf k\cdot$) refers to the notation system $\mathcal M$ of Section~\ref{secWeak}.
If $\mathcal F$ is a normalized notation system and $m\in \mathbb N$, we define ${\rm G}^\mathcal F_\infty(m)$ to be the least $\ell$ such that ${\rm G}^\mathcal F_\ell (m) = 0$, and ${\rm G}^\mathcal F_\infty(m) = \infty$ if no such $\ell$ exists.
More generally, $\goodp Fm\infty  r$ is the least $\ell$ such that $\goodp F m\ell r= 0$, if it exists.
We will show that $ \goodp Fm\infty r  \approx  H_{\mchb \omega r m}(r) $, so the left hand is finite since the right hand is.

To make this precise, we need to introduce an ordinal assignment for terms.
We note that the operations in $\mathcal M$ are well-defined on the ordinals, and as such we can consider expressions with base $\omega$.
For $\tau \in \mathbb M_k$, we define $\mchb \om {{}}\tau $ to be the result of replacing every occurrence of $k$ by $\omega$, i.e.~$\mchb \om {{}}0 =0$, $\mchb \om {{}}{(\tau+\sigma)} =\mchb \om{{}} \tau+ \mchb \om{{}} \sigma $, and $\mchb \om {{}} {k\tau} = \om \mchb \om {{}} \tau$.
If $m\in\mathbb N$, then $\mchb \omega k m := \val{\mchb\omega{{}} {\nf km} }$.

To continue, we need to calculate the normal form of $m-1$.
\shortproof{The following is easy to check.}

\begin{lemma}\label{lemmMminus}
Suppose that $0 < m \gknf k p+r$ and $k\geq 2$.
\begin{enumerate}

\item If $r>0$, then $m-1 \gknf k p+(r-1)$.

\item If $r = 0$, then $m-1 \gknf k (p-1)+(k-1)$.

\end{enumerate}
\end{lemma}

\longproof{
\begin{proof}
We work out the second item, as the first is similar.
Since $m=kp$, we clearly have $k(p-1) \leq m-1 < kp $, so the normal form of $m-1$ is of the form $k(p-1) + r'$, and plainly we must have $r' = k-1$.
\end{proof}
}

\shortproof{Inspection on Definition~\ref{defFS} then yields the following.}

\begin{lemma}\label{lemmMFS}
For every $m\in \mathbb N$ and $k\geq 2$,
\[\mchb \omega k {(m-1)} = (\mchb \omega km )[k-1] .\]
\end{lemma}

\longproof{
\begin{proof}
Write $m\gknf kp + s$ and consider two cases according to Lemma \ref{lemmMminus}.
If $s>0$, then $m-1\gknf kp+s-1$ and
\[(\mchb \omega km )[k-1] = \omega p + s-1 = \mchb \omega k {(m-1)} . \]
Otherwise, $m-1\gknf k(p-1) + (k-1)$ and
\[(\mchb \omega km )[k-1] = \omega (p-1) + k-1 = \mchb \omega k {(m-1)} . \]
\end{proof}
}

Below, we note that $\bch \tau {^\omega}$ is defined according to the base of $\tau$; if $\tau \in \mathbb M_k$, then $\bch \tau {^\omega}$ is obtained by replacing every occurrence of $k$ by $\omega$, and if $\tau \in \mathbb M_\ell$, then $\bch \tau {^\omega}$ is obtained by replacing every occurrence of $\ell$ by $\omega$.
\shortproof{A routine induction on term complexity shows that if $\tau \in \mathbb T^\mathcal M_k$ and $\ell > k$, then $\bch{\bch \tau {^\ell}}{^\omega}= \bch \tau{^\omega}.$
From this, we readily obtain the following.}

\longproof{

\begin{lemma}\label{lemmMultPsi}
Fix $k\geq 2$.
If $\tau \in \mathbb T^\mathcal M_k$ and $\ell > k$, then $\bch{\bch \tau {^\ell}}{^\omega}= \bch \tau{^\omega}.$
\end{lemma}

\proof
By induction on $\norm \tau$.
The claim is clear for $\tau \in \{ 0,1\}$.
If $\tau = \sigma + \rho$, then
\begin{align*}
 \bch{  \bch \tau {^\ell}}{^\omega} &= \bch{\bch{(\sigma+\rho)}{^\ell}}{^\omega}  = \bch{\bch \sigma {^\ell} }{^\omega} + \bch{\bch \rho {^\ell} }{^\omega} \\
 & =^{\text{\sc ih}} \bch{ \sigma  }{^\omega} + \bch{ \rho  }{^\omega} = \bch {(\sigma + \rho) }{^\omega} = \bch  \tau {^\omega}. 
\end{align*}

Otherwise, $\tau = k \cdot \sigma$ for some $\sigma$, and
then
\begin{align*}
\bch{ \bch {\tau }{^\ell}}{^\omega}   & = \bch{ \bch {(k\cdot \sigma) }{^\ell}}{^\omega}
= \bch  {(\ell \cdot \bch {\sigma }{^\ell})} {^\omega} \\
&= \om \cdot \bch{ \bch { \sigma } {^\ell} }{^\omega}
 =^{\text{\sc ih}} \om \cdot \bch   \sigma {^\omega}  = \bch \tau {^\omega}. & 
\end{align*}
\medskip
\endproof
}
\begin{proposition}\label{propMBCH}
If $2\leq k < \ell$, then
$\bch{  \bch m {\bases k\ell }}{\bases \ell \omega}= \bch m {\bases k\omega}.$
\end{proposition}

\proof Write $  \tau := \nf km $. By Lemma \ref{lemmExpPreserve}, $ \nf\ell m  =  \bch \tau {^\ell}  $.
Hence,
\[\bch m {\bases k \omega} =  \bch \tau {^\omega}  =  \bch{ \bch \tau {^\ell}} {^\omega} =  \bch { \bch m {\bases k\ell } }{\bases \ell \omega }.  \]
\endproof

\begin{theorem}\label{theoGoodmult}
For every $m\in \mathbb N$ and $k\geq 2$, $\goodpt Mmk$ is finite and
\[ \goodpt Mmk +k  =  H_{\mchb \omega k m } (k ). \]
\end{theorem}

\begin{proof}
We use transfinite induction below $\omega^\omega$ on $\mchb {{\omega}}{{k }}m $.
The base case, where $m=0$, yields $k $ on both sides.
Otherwise, we use Lemma~\ref{lemmMFS} to see that
\begin{align*}
\goodpt Mmk  &  = 1+ \goodpt M{\mchb {{k+1}}{{k }}m-1}{k+1}  \\
& \stackrel{\text{\sc{ ih}}}= 1+  H_{\mchb \omega {{k+1}} {(\mchb {{k+1}}{{k }}m-1 )}} (k +1 ) - k-1\\
&=
 H_{\fs{(\mchb \omega k m)}{k  } } (k +1 ) - k  &\text{(by \ref{lemmMFS} and \ref{propMBCH})}\\
 & =   H_{\mchb \omega k m } (k ) - k  ,
\end{align*}
where we are justified in using the induction hypothesis since
\[\mchb \omega {{k+1}} {(\mchb {{k+1}}{{k }}m-1 )} = (\mchb \omega k{\mchb {{k+1}}{{k }}m} )[k-1 ]  = (\mchb \omega k{m} )[k-1 ] <  \mchb \omega k{m} . \]
\end{proof}

\begin{corollary}
${\sf I}\Sigma_1$ does not prove that for every $m\in \mathbb N$, $\goodt Mm$ is finite.
\end{corollary}

\begin{proof}
Let $a_{x} = 2^{2x }$.
Then it is not hard to check that $\mchb \om 2 {a_{x}} = \om^{2x }$, and $H_{\om^{2x }}(2) > H_{\om^{x }}(x+1) = H_{\om^\om}(x)$.
But then, by Theorem~\ref{theoGoodmult}, $\goodt M{a_x} + 2 > H_{\om^\om}(x)$.
It follows from the second item of Theorem \ref{theoHIndep} that ${\sf I}\Sigma_1$ does not prove that $\goodt  M{a_x} $ is finite for all $x$.
\end{proof}

\subsection{Termination of the classic Goodstein process}

Now we turn our attention to the classic Goodstein process based on $\mathcal E$.
Recall that in this context, $r\cdot\tau$ is shorthand for $\tau+\ldots +\tau$, $r$ times.

\reread

\begin{lemma}\label{lemmMCBound}
If $k\geq 2$ and $\tau$ is in base $k$ normal form then $\mc{\mchb \omega k \tau }<k$.
\end{lemma}

\begin{proof}
It suffices to observe that $r\cdot k^a + b$ cannot be in normal form for any $r\geq k$, since otherwise
\[r\cdot k^a + b \geq k^{a+1}. \]
With this and an easy induction on term complexity, we see that no term in normal form may contain coefficients greater than or equal to $k$.
\end{proof}

\shortproof{The following may readily be checked by induction on $m$, by writing $m = k^a+b$ in normal form and comparing $\mchb \omega {k} (m -1) $ to $\fs {(\mchb \omega {{k}} m  )}{k-1}$ according to Definition~\ref{defFS}.}

\begin{lemma}\label{lemmFSBound}
If $k\geq 2$ and $0< m\in\mathbb N$, then
\[ \fs {(\mchb \omega {{k }} m  )}{k-1} \leq  \mchb \omega {k } (m -1) <   \mchb \omega {k } m.\]
\end{lemma}

Note that in contrast to Lemma~\ref{lemmMFS}, we do not always obtain equality on the left, but this is enough to obtain a lower bound.
It is also worth remarking that  $\mchb \omega {k } (m -1) <   \mchb \omega {k } m$ for all $m>0$ is equivalent to the statement that if $n<m$, then $\mchb \omega {k } n <   \mchb \omega {k } m$ and thus the $\om$-base change is monotone, just as the finitary ones.
This lemma will allow us to compare Hardy hierarchies with the length of the standard Goodstein process.

\longproof{
\begin{proof}
By induction on $m$.
Write $\mchb {}{}{}$ for $\mchb {\omega}{k+1}{} $ and $m = (k+1)^a+b$ in normal form.
The claim is trivial for $m=0$, so we assume otherwise and consider the following cases.
\begin{Cases}
\item ($b>0$).
Then, $m -1  = k^\tau+(\sigma-1)$, and the induction hypothesis readily yields\footnote{In principle, the last inequality may be strict if $\omega^{\mchb {}{} a} +  \mchb {}{} b   $ is not in Cantor normal form; however, it can be proven that if $\tau$ is in normal form, then $\mchb {}{}\tau$ is in CNF.}
\[\mchb {}{} (m -1) = \omega^{\mchb {}{} a} + \mchb {}{} (b -1) \stackrel{\text{\sc ih}} \geq \omega^{\mchb {}{} a} + \fs {(\mchb {}{} b  )}k \geq   \fs {(\mchb {}{} m  )}k  ,\]
as well as
\[\mchb {}{} m =\omega^{\mchb {}{} a} + \mchb {}{}  b \stackrel{\text{\sc ih}}  > \omega^{\mchb {}{} a} +  \mchb {}{} (b -1) = \mchb {}{} (m -1).\]
\item ($b=0$).
Consider the following sub-cases.
\begin{Cases}

\item $(a=0) $. Then $m-1 = 0$ and $\mchb {}{} m
=1$, so
\[\mchb {}{} (m -1)
=0
= \fs {(\mchb {}{} m  )}k < \mchb {}{} m.\]

\item $(a>0)$.
By Lemma~\ref{lemmExpNFProp}.\ref{itExpNFPRopOne}, $m-1=(k-1)^a-1 \nnf \sum_{i=0}^{a-1} k\cdot (k+1)^{i}    $, so $\mchb{}{}(m-1) = \sum_{i=0}^{a-1} \omega ^{i}\cdot k  \geq \omega ^{a-1}\cdot k$.
By the induction hypothesis $\mchb {}{} a > \mchb {}{} (a-1)$, so
\begin{align*}
\mchb {}{} m & = \omega^{\mchb {}{}  a } > \omega^{\mchb {}{}  (a-1)}(k+1)> \mchb {}{} (m -1) \\
& \geq \omega^{\mchb {}{} (a -1)}\cdot  k \stackrel{\text{\sc ih}} \geq \omega^{\fs{(\mchb {}{} a)} k} \cdot  k \geq \fs{(\mchb {}{} m)} k,  
\end{align*}
where the last inequality can be checked to hold whether or not $\mchb {}{} a$ is a limit by inspection on Definition~\ref{defFS}.
\end{Cases}
\end{Cases}
\end{proof}
}
 
\begin{theorem}\label{theoGrowthExp}
For all $m\in\mathbb N$ and $x\geq 2$,
\[H_{\mchb \omega {x} m}(x ) \leq \goodp Em\infty x +x  \leq H_{\mchb \omega x m}(x+1 ). \]
\end{theorem}

\begin{proof}
The upper bound is immediate from Lemmas~\ref{lemmMCBound} and~\ref{lemmFSBound} and Proposition~\ref{propHUpper}, while the lower bound follows from Lemma~\ref{lemmFSBound} and Proposition~\ref{propHLower}.
\end{proof}

\begin{corollary}
$\sf PA$ does not prove that for every $m\in \mathbb N$, $\goodt Em$ is finite.
\end{corollary}

\begin{proof}
If this were provable in $\sf PA$, it would be provable in ${\sf I}\Sigma_n$ for some $n$.
For $k>0$, it would follow that the function $f$ given by $f(x) = \goodt E{2_{n+x+1}}+2$ is total, where $x_y$ denotes the superexponential function.
But $\bch {2_{n+x+1} }{\bases {2}\omega} = \omega_{n+x+1}$, and it is easy to check using Theorem~\ref{theoOrdMon} and an easy induction that $H_{\omega_{n+x+1}}(2) > H_{\omega_{n+1}}(x)$, so $f(x)>H_{\omega_{n+1}}(x)$, contradicting the second item of Theorem~\ref{theoHIndep}.
\end{proof}

\done

\section{Goodstein Walks}\label{secWalks}

In this section we introduce and study Goodstein walks.
These are Goodstein-like processes which are defined independently of a normal form representation; natural numbers may be written in an arbitrary way using the functions from $\mathcal F$.
Aside from this, the definition is analogous to that of standard Goodstein processes.

\begin{definition}
Fix a notation system $\mathcal F$.
A {\em Goodstein walk} (for $\mathcal F$) is a sequence $(m_i)_{i <\alpha}$, where $\alpha \leq \infty$, such that for every $i<\alpha$, there is a term $\tau _i \in \mathbb T^\mathcal F_{i+2}$ with $\val{\tau_i} = m_i$ and $m_{i+1} = \bch \tau {^{i+3}} - 1$.
\end{definition}

\begin{theorem}\label{theoWalk}
Let $\mathcal F$ be a normalized notation system with $+$ and $1$.
Suppose that $\mathcal F$ is base-change maximal, and that for every $m\in \mathbb N$ there is $i\in \mathbb N$ such that ${\rm G}^\mathcal F_i (m) = 0$.
Then, every Goodstein walk for $\mathcal F$ is finite.
\end{theorem}

\proof
Let $\mathcal F$ satisfy the assumptions of the theorem and $(m_i)_{i=0}^\alpha$ be a Goodstein walk for $\mathcal F$.
Let $m = m_0$.
By induction on $i$, we check that $m_i \leq {\rm G}^\mathcal F_i (m)$.
For the base case this is clear. Otherwise, $m_{i+1} = \val{\bch{\tau_i}{^{i+3}}} - 1$ for some term $\tau_i \in \mathbb T^\mathcal F_{i+2}$, and thus
\[m_{i+1} = \val{\bch{\tau_i}{^{i+3}}} - 1 \leq \bch{m_i}{^{i+3}_{i+2}} - 1 \stackrel{\text{\sc ih}}\leq \bch{{\rm G}^\mathcal F_i (m)}{^{i+3}_{i+2}} - 1 = {\rm G}^\mathcal F_{i+1} (m),\]
where the second inequality uses Proposition \ref{propMaxToMon} and the assumption that $\mathcal F$ is base-change maximal.
Thus if we choose $i$ such that ${\rm G}^\mathcal F_i (m) =0 $, we must have $\alpha \leq i $.
\endproof

As a corollary, we obtain the following extension of Goodstein's theorem.

\begin{theorem}\label{theoWalk}
Any Goodstein walk for $\mathcal M$, $\mathcal E$, or $\mathcal L$ is finite.
\end{theorem}

\begin{example}
Consider alternative normal forms based on $\mathcal L$ as follows. Let $k\geq 2$ and $m\geq 0 $.
First, set $\nf k0 = 0$.
For $m > 0$, let $p_ 1 \cdots p_n$ be the decomposition of $m$ into prime factors.
If $n\leq 1 $, we write $m = k^r + b $ with $m<k^{r+1}$ (as in the exponential normal forms) and set $\nf km = k^{\nf kr} + \nf kb$.
Otherwise, set $\nf k m = \nf k{p_1} \cdots \nf k{p_n}$.

These normal forms do not have the natural structural properties that are useful in a direct proof of termination.
For example, $7 \nfpar 3 3^1 + 3^1 + 1$ and $8 \nfpar 3 2\cdot 2\cdot 2$.
It follows that $\bch 7{\bases 34} = 4^1 + 4^1 + 1 = 9 $ and $\bch 8{\bases 34} = 2\cdot 2\cdot 2 = 8 $.
Thus the base-change operator is not monotone, in the sense that the analogue of Corollary \ref{corExpMon} fails.
Similarly, the natural ordinal assignment would not be monotone, as it would yield $\bch 7 {\bases 3\omega} = \omega \cdot 2 + 1 $ and $\bch 8 {\bases 3\omega} = 8$.
Without these monotonicity properties, a termination proof as given in Section \ref{secWeak} would not go through.
Nevertheless, the Goodstein process based on these normal forms is terminating by Theorem \ref{theoWalk}, and such a direct proof is not needed.
\end{example}

\section{Phase transitions}\label{secPhase}

We have defined general Goodstein sequences $\goodp {{}}mir$, where $r\geq 2$ is the base of the first term.
In this section, we aim to find weakenings of this statement provable in ${\sf I}\Sigma_n$.
The strategy is to bound the value of $n$, but for any fixed $r$ and $N$, the statement $\forall m<N (\goodp {{\mathcal E}} m\infty r<\infty)$ is provable in ${\sf I}\Sigma_1$ (or even weaker systems) since it can be proved by checking finitely many instances.
Instead, we may let $r$ vary, and moreover have $N$ depend on $r$.
Specifically, we will set $N=r_k$, where we recall that $x_y$ denotes the superexponential function.
The provability of the termination of such restricted Goodstein processes in ${\sf I}\Sigma_n$ depends on whether $k\leq n$.

\begin{theorem}\label{theoPhase}
For $k\geq 1$, let $\varphi_k$ be the statement
\begin{quote}
For every $r\geq 2$ and every $m<r_k$, $\goodp E {m}\infty {r}$ is finite.
\end{quote}
Then, for every $n,k\geq 1$, ${\sf I}\Sigma_n\vdash \varphi_k$ if and only if $k \leq n $.
\end{theorem}

\begin{proof}
Fix $n,k\geq 1$.
Let $r \geq 2$ and $m<r_k$.
By Lemma~\ref{lemmFSBound}, $\mchb \omega r m < \mchb \omega r{r_k} = \omega_k$.

By Theorems~\ref{theoGrowthExp} and~\ref{theoOrdMon}, along with $\mc{ \mchb \omega r m}<r$ by Lemma~\ref{lemmMCBound}, $\goodp E m\infty {r}<H_{ \mchb \omega r m } (r+1) < H_{\omega_k}(r+1 )$.
Moreover, we established these inequalities using elementary means, so they are provable in ${\sf I}\Sigma_n$.
If $k \leq n $, from the provable totality of $H_{\omega_k}$ in ${\sf I}\Sigma_n$ (Theorem~\ref{theoHIndep}), we conclude that $\varphi_k$ is provable in ${\sf I}\Sigma_n$.

If $k>n $, let $m=(r+1)_{k}-1$.
Note that $H_{\om_k}(r)=H_{\fs{\om_k} r }(r+1) $.
By Lemma~\ref{lemmFSBound}, $\mchb \omega {r+1} m \geq \fs{\om_k} r =  \om_k r$.
Since $\mc{\om_k r}=r$, in view of Theorems~\ref{theoOrdMon} and~\ref{theoGrowthExp}, it follows that $H_{\fs{\om_k} r}(r  ) < H_{\bch m{\bases {r+1}\om}}(r+1)\leq  \goodp E {m}\infty {r+1} +r+1$.
Thus the function
\[r\mapsto \goodp E {(r+1)_k-1}\infty {r+1}+r+1\]
grows faster than $H_{\omega_{n + 1}}(r  )$, hence once again by Theorem~\ref{theoHIndep}, the statement $\forall r \ \big ( \goodp E {r_{k}-1}\infty {r}<\infty \big )$ is not provable in ${\sf I}\Sigma_n$.
Since clearly $ {r_{k}-1}<r_k$, neither is $\varphi_k$. 
\end{proof}

\section{Concluding remarks}
We have explored two notions of `optimality' for notations for Goodstein processes.
The first, norm minimality, is naturally motivated, but we have seen that it may fail for some otherwise well-behaved normal forms; specifically, for the notation system $\mathcal L$ and the standard Goodstein normal forms.
The second, base-change maximality, is perhaps more subtle but leads to some very interesting consequences: most notably, it allows for `normal form-free' Goodstein processes.
In the context of $\mathcal M$ and $\mathcal E$, normal forms are simple enough that the benefit of eliminating them is debatable.
However, already for elementary terms we have that the standard normal forms are not norm-minimizing, and it is unclear if norm-minimizing terms will lead to a Goodstein process with a natural ordinal interpretation.
Thus it is surprising that such an ordinal interpretation -- or even identifying norm-minimizing normal forms to begin with -- is not needed to establish the termination of {\em any} Goodstein processes based on this notation system.

However, for more powerful notation systems, the elimination of normal forms becomes more pressing, and here the base-change maximality technique is crucial.
For notation systems based on the Ackermann function~\cite{AckermannGoodstein}, normal forms are already quite cumbersome, so normal form-free Goodstein principles would lead to substantially more accessible independent statements.
We have already applied base-change maximality to fast-growing hierarchies~\cite{FastGoodstein}, where once again normal forms can be quite complex.
We believe that, going forward, the analysis of norm-minimality and, especially, base-change maximality will become an essential ingredient in the study of new and ever-more-powerful Goodstein principles.

\subsection*{Acknowledgements}

The authors would like to thank the anonymous referees who made many useful suggestions, including spotting several mathematical errors in previous versions of this article.

The authors of this work were supported by the FWO-FWF Lead Agency grant G030620N (FWO)/I4513N (FWF).


%\bibliographystyle{asl}
%\bibliography{biblio}

\end{document}